\documentclass[11pt]{amsart}
\usepackage{amsmath,amsthm,amsfonts,amssymb,times}

\usepackage{titlesec}
\titleformat{\section}  %% large section titles
{\normalfont\Large\bfseries}
{\thesection}{1em}{}
\titleformat{\subsection}
{\normalfont\large\bfseries}
{\thesubsection}{1em}{}

\numberwithin{equation}{section}  %%% equations numbers (A.B)

\DeclareMathAlphabet{\curly}{U}{rsfs}{m}{n}  %% curly font

\theoremstyle{remark}

\theoremstyle{plain}
\newtheorem{prop}{Proposition}
\newtheorem{lem}{Lemma}[section]
\newtheorem{thm}{Theorem}
\newtheorem{cor}{Corollary}
\newtheorem{conjecture}{Conjecture}

\numberwithin{equation}{section}

\newcommand{\ZZ}{{\mathbb Z}}
\newcommand{\QQ}{{\mathbb Q}}

\newcommand{\NN}{{\mathbb N}}

\makeatletter
\renewcommand{\pmod}[1]{\allowbreak\mkern7mu({\operator@font mod}\,\,#1)}
\makeatother

\newcommand{\bal}{\[\begin{aligned}}
\newcommand{\eal}{\end{aligned}\]}

\newcommand{\be}{\begin{equation}}
\newcommand{\ee}{\end{equation}}

\newcommand{\ssum}[1]{\sum_{\substack{#1}}}  %%% stacked sum
\newcommand{\sprod}[1]{\prod_{\substack{#1}}}  %%% stacked product

\newcommand{\lam}{\ensuremath{\lambda}}

\newcommand{\del}{\ensuremath{\delta}}
\newcommand{\eps}{\ensuremath{\varepsilon}}

\newcommand{\er}{\text{e}}
\renewcommand{\le}{\leqslant}

\renewcommand{\ge}{\geqslant}

\renewcommand{\(}{\left(}
\renewcommand{\)}{\right)}
\newcommand{\ds}{\displaystyle}
\newcommand{\pfrac}[2]{\left(\frac{#1}{#2}\right)}  %%% frac with
\newcommand{\PP}{\mathcal{P}}

\newcommand{\pz}{\ensuremath{p_0^{\phantom{1}}}}  %%% nicely formatted p_0
\newcommand{\qz}{\ensuremath{q_0^{\phantom{1}}}}  %%% nicely formatted q_0
\newcommand{\qone}{\ensuremath{q_1^{\phantom{1}}}}  
\newcommand{\pone}{\ensuremath{p_1^{\phantom{1}}}}  
\newcommand{\mone}{\ensuremath{m_1^{\phantom{1}}}}  
\newcommand{\aone}{\ensuremath{a_1^{\phantom{1}}}}  
\newcommand{\az}{\ensuremath{a_0^{\phantom{1}}}}  

\begin{document}

\title{Sieving very thin sets of primes, and Pratt trees with missing primes}
%\title{On certain recursively defined sets of primes}
\author{Kevin Ford}
\date{\today}
\address{Department of Mathematics, 1409 West Green Street, University
of Illinois at Urbana-Champaign, Urbana, IL 61801, USA}
\email{ford@math.uiuc.edu}
\begin{abstract}
Suppose $\PP$ is a set of primes, such that for every $p\in\PP$, every prime factor of $p-1$ is
also in $\PP$.  We apply a new sieve method to 
show that either $\PP$ contains all of the primes or the counting function of $\PP$ 
is $O(x^{1-c})$ for some $c>0$, where $c$ depends only on the smallest  prime not in $\PP$.
Our proof makes use of results connected with Artin's primitive root conjecture.
\end{abstract}

\dedicatory{Dedicated to the memory of Paul T. Bateman}

\maketitle

%%%%%%%%%%%%%%%%%%%%%%%%%

\section{Introduction}

Consider a set $\PP$ of primes satisfying the condition:
\be\label{divisorclosed}
p\in \PP \, \implies \, \forall q|(p-1), q\in \PP.
\ee
Here and throughout, the letters $p$, $q$ and $r$ denote primes.  
Trivial examples of sets $\PP$ are the empty set and the set of all primes.

We are concerned in this note with nontrivial examples, nonempty $\PP$ 
omitting at least one prime (since $2\not\in \PP$ implies that $\PP$ is empty, 
the smallest omitted prime must be odd).  Let $\pz$ denote the smallest prime \emph{not} in $\PP$
and let $P(x)=\# \{p\in \PP : p\le x\}$ be the associated counting
function.  Our main result is the following.

\begin{thm}\label{mainthm}
Let $\PP$ be a set of primes satisfying \eqref{divisorclosed} that does not contain the prime $\pz$.  
There are constants $\del>0$ and $c>0$, depending only on $\pz$, such that $P(x)\le c x^{1-\del}$.
\end{thm}

Theorem \ref{mainthm} implies that either $\PP$ is the set of all primes
or $\PP$ is a very ``thin'' set of primes.  The elements of $\PP$ have the property that
for every prime $p\not\in\PP$, $\PP$ omits the residue classes $0,1 \mod p$.
Standard application of sieve methods produce only the much weaker bound $P(x)\ll x/\log^2 x$ 
(see Proposition \ref{simplesieve} below).  The weakness stems from the fact that sieve methods
ignore congruential restrictions for ``large'' primes (i.e., those primes $>\sqrt{x}$, when 
bounding the number of elements of a set that are $\le x$).  With our new method, we are able
exploit these large prime restrictions.

To the author's knowledge, sets of primes satisfying \eqref{divisorclosed} were first considered by
R. D. Carmichael \cite{C1,C2} in his work on the conjecture that now bears his name.
Here $\phi$ is Euler's ``totient'' function.

\begin{conjecture}[Carmichael's Conjecture]
 For every positive integer $a$, there is an positive integer $b\ne a$ such that $\phi(b)=\phi(a)$.
\end{conjecture}

The conjecture remains open, although the smallest counterexample $a$, if there is one, is known to
exceed $10^{10^{10}}$ \cite{Ford98}.  
Assuming a counterexample $a$ exists, Carmichael \cite{C2} attacked the problem with the 
following simple result.
\be\label{carmichael}
 \text{If } d\prod_{p|d} p \text{ divides } a \text{ and } p=1+d \text{ is prime, then } p^2|a.
\ee

Applying \eqref{carmichael} successively with $d=1, 2, 6$ and $42$,
it follows immediately that $2^2 3^2 7^2 43^2 | a$.  From here, Carmichael considers two cases: (i)
$3^2 \| a$, which easily implies $13^2|a$, and (ii) $3^3|a$, which implies
by \eqref{carmichael} that $19^2|a$.  In each case, one can use \eqref{carmichael}
to produce many more primes which divide $a$.  More precisely, in case (i), $a$ must be divisible by
all primes in $\PP'$, where $\PP'$ contains $2,3,7,13,43$ and for other primes $p$, $p\in\PP'$ if and only
if $p-1$ is squarefree and for every $q|(p-1)$, $q\in \PP'$.  Then 
$\PP'=\{2, 3, 7, 13, 43, 79, 547, 3319, 6163,\ldots\}$.
Similarly, in case (ii), $a$ is divisible by every prime in $\PP''$, where $2, 3, 7, 19, 43$ are in $\PP''$
and for other primes $p$, $p\in\PP''$ if and only if (a) $p-1$ is either squarefree or $3^2|(p-1)$ and 
$\frac{p-1}{9}$ is squarefree and (b) for every $q|(p-1)$, $q\in \PP''$. 
Then $\PP''=\{2, 3, 7, 19, 43, 127, 2287, 4903, 5419, \ldots\}$.
 Thus, the sets $\PP'$ and $\PP''$
each satisfy \eqref{divisorclosed} and omit the prime 5.  By Theorem \ref{mainthm}, each of $\PP'$ and
$\PP''$ has counting function satisfying $P(x)\ll x^{1-c}$ for some $c>0$.  Carmichael's conjecture
follows if both $\PP'$ and $\PP''$ are infinite.

In a similar spirit, Pomerance \cite{Pom74} showed that if $x$ satisfies $p^2|x$ 
whenever $(p-1)|\phi(x)$, then there is no number $b\ne x$ with $\phi(b)=\phi(x)$.  However,
Pomerance argued heuristically that no $x$ with this property exists.

\medskip

Sets satisfying \eqref{divisorclosed} also arise in the distribution of iterates of Euler's function. 
Let $\phi_k(n)$ denote the $k$-th iterate of $\phi$ (e.g., $\phi_2(n)=\phi(\phi(n))$),
and let $F(n)=\prod_{k\ge 1} \phi_k(n)$ (the product is finite, since $\phi_k(n)=1$ for large $k$).
Divisibility properties of $F(n)$ were considered by Luca and Pomerance \cite{LP} in connection
with construction of irreducible, radical extensions of $\QQ$.  The prime factors of $F(p)$, where $p$
runs over the primes, were considered by Bayless \cite{Bay}.
Further results on $\phi_k(n)$ may be found in \cite{EGPS}.

\begin{cor}
For every prime $r\ge 3$, there is a constant $s<1$ so that $\# \{ n\le x : r\nmid F(n) \} \ll x^s$.
\end{cor}

\begin{proof}
Let $\PP_r$ be the largest set satisfying \eqref{divisorclosed} such that $r\not\in\PP_r$; i.e., $\PP_r$ contains
all primes less than $r$, $r\not\in\PP_r$ and a prime $p>r$ lies in $\PP_r$ if and only if for all $q|(p-1)$,
$q\in\PP_r$.   For example,
\begin{align*}
\PP_3 &= \{2, 5, 11, 17, 23, 41, 43, 83, 89, 101, 137, 167, 179, 251, 257, \ldots\},\\
\PP_5 &= \{2, 3, 7, 13, 17, 19, 29, 37, 43, 53, 59, 73, 79, 97, 103, \ldots \}. 
\end{align*}
 Since $\phi(p^a)=p^{a-1}(p-1)$, for each $n$ the (finite) set $\PP$ of prime factors of 
$F(n)$ satisfies \eqref{divisorclosed}.
Hence, if $r\nmid F(n)$, then all the prime factors of $n$ belong to $\PP_r$.  By Theorem \ref{mainthm}, 
for some $c>0$, depending on $r$, there are $\ll x^{1-c}$ primes in $\PP_r$ that are less than $x$.  For any $s>1-c$, it follows by partial summation
that the number of $n\le x$ with $r\nmid F(n)$ is at most
\[
 \sum_{p|n\implies p\in \PP_r} \pfrac{x}{n}^s = x^s \prod_{p\in\PP_r} \(1-p^{-s}\)^{-1} 
\le x^s \exp \Big\{ \sum_{p\in\PP_r} \frac{1}{p^s-1} \Big\} \ll_{r,s} x^s. \qedhere
\]
\end{proof}

The set $\PP_r$ is also the set of all primes $p$ for which the \emph{Pratt tree} for
$p$ has no node labeled $r$.  The Pratt tree for a prime $p$ is recursively define as
the tree with root labeled $p$, and below $p$ are links to the Pratt trees of each $q|(p-1)$.
Properties of Pratt trees (e.g. the distribution of the height $H(p)$, number of nodes, etc.) were 
extensively studied in \cite{FKL}.  In alternative terminology, $\PP_r$ is the set of primes $p$
for which there is no \emph{prime chain} $r\prec p_0 \prec \cdots \prec p_k \prec p$, where
$a\prec b$ means $b\equiv 1\pmod{a}$, and $p_0,\ldots,p_k$ are primes.

\medskip

A finite group $G$ is said to have Perfect Order Subsets (POS) if the \emph{number} of elements 
of $G$ of any given order divides $|G|$.  This notion was introduced by 
Finch and Jones \cite{FJ} in 2002.  In the case of Abelian groups, Finch and Jones reduced the problem
of determining which groups have POS to studying those of the form $G=(\ZZ/p_1\ZZ)^{a_1} \times \cdots \times (\ZZ/p_j\ZZ)^{a_j}$,
where $p_1,\ldots,p_j$ are distinct primes.  This group has POS if and only if $f(n)|n$, where $n=p_1^{a_1}\cdots
p_j^{a_j}=|G|$ and $f(n)=\prod_{p^a\| n} (p^a-1)$.  Suppose that $3\nmid n$.  It follows quickly that the 
primes dividing $n$ must lie in $\PP_3$.
Developing explicit estimates for the counting function
of $\PP_3$, it was shown in \cite{FKL2} that an Abelian group with POS is either $\ZZ/2\ZZ$ or has order
divisible by 3. This answered a question posed in \cite{FJ}.

It is intractible with existing methods to prove nontrivial lower bounds for $P(x)$ even in the ``easiest'' cases
when $\PP=\PP_r$.  The difficulty, now evident from Theorem \ref{mainthm}, is to show that many primes exist
with the prime factors of $p-1$ restricted to a very thin set.

\begin{conjecture}
 Each set $\PP_r$ is infinite.  
\end{conjecture}

This conjecture follows, for instance, if there are inifinitely many primes of the form $2^a3^b+1$
and infinitely many primes of the form $2^a5^b+1$.  Each of these latter statements appears to 
be plausible, based on computations.

The author computed the elements of $\PP_3$ up to $2^{44}$ ($\approx 1.7\times 10^{13}$) for use in
\cite{FKL2}, and $P(x) \approx x^{0.62}$ in this range.  The recursive nature of
the sets $\PP$, however, does not lead to any natural heuristic argument for the size of $P(x)$.
The growth appers to be highly dependent on which small primes are omitted from the set.  For an extreme 
example, consider $\PP$ to be the ``largest '' set omitting the primes $3,5,17,257,65537$ 
(the list of known Fermat primes -- primes that are 1 more than a power of 2).  
It is a famour unsolved problem whether or not there are additional Fermat primes.
If there are no further Fermat primes then $\PP=\{2\}$, while
if another Fermat prime exists then $\PP$ could potentially be infinite.

Based partly on the computations for $\PP_3$, we make an educated guess for the growth of $P(x)$.

\begin{conjecture}
 For each $r$, there is a number $\delta_r>0$ such that $P(x)=x^{1-\delta_r+o(1)}$ as $x\to\infty$.
\end{conjecture}

We further guess that $\delta_r\to 0$ as $r\to\infty$.

\medskip

\textbf{Outline of the paper.}  The next section contains relatively simple estimates for $P(x)$
which are needed to bootstrap the more complicated iterative method in Sections 3 and 4.
Basically, we find recursive inequalities for the density of primes whose Pratt tree has height $\le j$,
for $j=0,1,2,\ldots$.
The main iteration inequalities are proved in Section 3, together with a conditional result that
implies Theorem \ref{mainthm} under the assumption that a certain matrix has eigenvalues all inside
the unit circle.  Section 4 concludes the proof of Theorem \ref{mainthm} by showing that indeed the matrix 
has this property.  Our method
uses results from the circle of ideas used to attack Artin's primitive root conjecture.

%%%%%%%%%%%%%%%%%%%%%%%%%%%%%%%%%%%%%%%%%%%%%%%%%%%%%%
%
\section{Simple sieve estimates}
%
%%%%%%%%%%%%%%%%%%%%%%%%%%%%%%%%%%%%%%%%%%%%%%%%%%%%%%

From now on, we always assume that $\PP$ is a set of primes satisfying \eqref{divisorclosed}
and that there is some prime not in $\PP$, the smallest such we denote by $\pz$.
All estimates using the Landau $O-$symbol and Vinogradov $\ll-$symbol may depend on $\pz$, but
not on any other quantity.  The symbols $p$ an $q$, with or without subscripts, always denote primes.

\begin{prop}\label{simplesieve}
We have $P(x) \ll x/\log^2 x$ and $\ds \sum_{p\in \PP} \frac{1}{p} \ll 1$.
\end{prop}

\begin{proof}
By \eqref{divisorclosed} and standard application of sieve methods \cite[Theorem 4.2]{HR},
\be\label{Pxupper}
 P(x) \ll \frac{x}{\log x} \sprod{q\le x^{1/4}\\q\not\in\PP}\(1-\frac{1}{q}\). 
\ee
Since $\PP$ omits all primes  $q\equiv 1\pmod{\pz}$, \eqref{Pxupper} and Mertens' estimate for 
primes in arithmetic progressions imply that
\[
 P(x)\ll \frac{x}{(\log x)^{1+1/(\pz-1)}}.
\]
By partial summation, $\sum_{p\in\PP} 1/p \ll 1$ and
thus $\prod_{p\in\PP} (1-1/p)\gg 1$. 
Applying \eqref{Pxupper} again and Mertens' bound, we find that $P(x)\ll x/\log^2 x$.
\end{proof}

Our proof of Theorem \ref{mainthm} requires a slight improvement to Proposition \ref{simplesieve}
to bootstrap the method.

\begin{lem}\label{sieve3}
 We have $P(x) \ll x(\log x)^{-5/2}$.
\end{lem}

\begin{proof}
For $p\in \PP$ with $p\le x$, let $q$ be the largest prime factor
of $p-1$, write $p=1+qm$ and define $y=x^{1/(10\log\log x)}$.  By standard counts of smooth numbers
(see e.g., Theorem 1 in \S III.5 of \cite{Ten}), the number of $p$ with $q\le y$ is
$\ll x/\log^5 x$.  Next, fix $m\le x/y$, and observe that $q\not\equiv 1\pmod{r}$ for each prime $r\not\in\PP$.
By sieve methods \cite[Theorem 4.2]{HR} and Proposition \ref{simplesieve}, 
the number of $p\le x$ is bounded above by
\begin{align*}
\#\{ n\le x/m : \forall r\not\in \PP, r\nmid n(n+1)(mn+1) \} &\ll 
\frac{x/m}{\log^3(x/m)} \prod_{p|m(m+1)} \(1-\frac{1}{p}\)^{-1} \\
&\le \frac{x/m}{\log^3 (x/y)} \frac{m^2+m}{\phi(m^2+m)} \\
&\ll \frac{x(\log\log x)^4}{m\log^3 x}.
\end{align*}
Since $m$ is composed of prime factors in $\PP$, Proposition \ref{simplesieve} implies
\[
 \sum_{m} \frac{1}{m} \le \prod_{p\in \PP} \(1-\frac{1}{p}\)^{-1} \ll 1.
\]
The claimed bound follows.
\end{proof}

%%%%%%%%%%%%%%%%%%%%%%%%%%%%%%%%%%%%%%%%%%%%%%%%%%%%%%
%
\section{The main iteration}
%
%%%%%%%%%%%%%%%%%%%%%%%%%%%%%%%%%%%%%%%%%%%%%%%%%%%%%%

The proof of Theorem \ref{mainthm} is based on recursive inequalities for sums over 
subsets of $\PP$.  We partition the primes $p\in \PP$
according to the height $H(p)$ of their Pratt trees.  The height may be defined iteratively by 
\[
 H(2)=1, \qquad H(p) = 1 + \max_{q|(p-1)} H(q).
\]
We denote
\[
 \PP_h = \{ p\in \PP : H(p) \le h \} \qquad (h\in\NN)
\]
and also define, for $h\in \NN$ and real $s>0$,
\[
 V_h(s) = \sum_{p\in \PP_h} \frac{1}{(p-1)^s}, \qquad T_h = \{ n\in \NN : p|n \implies p\in \PP_h \}.
\]
We also allow $h=\infty$ in the above notations.  In particular, $\PP_1=\{2\}$ and $\PP_\infty=\PP$.
A trivial, but very useful observation, is that
\be\label{phth-1}
p\in \PP_h \, \implies \, p-1 \in T_{h-1}.
\ee
Our goal is to show that $V_{\infty}(s)$ is finite for some $s<1$, which is clearly equivalent to
Theorem \ref{mainthm}.

A trivial bound which we will use often is
\be\label{sumq}
\sum_{a=1}^\infty \frac{1}{q^{as}} = \frac{1}{q^s-1} \le \frac{\lam(s)}{(q-1)^s}, \quad
\lam(s) = \frac{1}{2^s-1} \qquad (0<s\le 1).
\ee

\begin{lem}\label{V0}
 For every $h\ge 1$, $V_h(s)$ is continuous for $0<s\le 1$.
\end{lem}

\begin{proof}
 It suffices to show that $V_h(s)$ is finite.  This follows by induction on $h$, starting from
$V_1(s) = 1$ for all $s$, and using \eqref{phth-1} and \eqref{sumq} to obtain the iterative bound
\be\label{VhVh-1}
\begin{split}
 V_h(s) & \le \sum_{m\in T_{h-1}} \frac{1}{m^s}  = \prod_{p\in \PP_{h-1}} \frac{1}{1-p^{-s}}
= \prod_{p\in \PP_{h-1}} \(1 + \frac{1}{p^s-1}\) \\
&\le \prod_{p\in \PP_{h-1}} \(1 + \frac{\lam(s)}{(p-1)^s}\) \le \er^{\lam(s) V_{h-1}(s)}.
\end{split}
\qedhere
\ee
\end{proof}

\newcommand{\Tt}{\widetilde{T}}
\newcommand{\Tb}{\overline{T}}
We next develop more sophisticated bounds for $V_h(s)$ in terms of $V_{h-1}(s)$.  
It turns out that when $s$ is close to 1, $V_h(s)$ is dominated by primes $p\in \PP_h$
for which $p-1$ has only a single ``large'' prime factor (meaning a prime $q$ with 
large height $H(q)$).  For $k > j \ge 1$, denote
\bal
\Tt_{j,k}&=\{n\in T_k\setminus T_j:p^2|n\text{ for some }p\in\PP_k\setminus \PP_j, \text{ or } n \text{ has at least
2 prime factors in } \PP_k\setminus \PP_j\}, \\
\Tb_{j,k}&=(T_k\setminus T_j)\setminus \Tt_{j,k}.
\eal

\begin{lem}\label{V1}
 For $k>j\ge 1$ and $s>0$, we have
\[
 \sum_{n\in \Tt_{j,k}} \frac{1}{n^s} \le 2\lam(s)^2 \( V_k(s) - V_j(s) \)^2  \er^{\lam(s) V_k(s)}.
\]
\end{lem}

\begin{proof}
For $n\in \Tt_{j,k}$, let $\qone,\ldots,q_d$ be the prime factors of $n$ that are in 
$\PP_k \setminus \PP_j$ (``large'' prime factors).  Then $n=q_1^{a_1}\cdots q_d^{a_d} m$, where $m\in T_j$
and $a_1,\ldots,a_d$ are positive integers.
Also, either (i) $d\ge 2$ or (ii) $d=1$ and $q_1^2|n$.  We deduce that
\[
\sum_{n\in \Tt_{j,k}} \frac{1}{n^s} \le \Bigg[ \sum_{q\in\PP_k \setminus \PP_j} \sum_{a=2}^\infty \frac{1}{q^{as}} +
\sum_{d=2}^\infty \frac{1}{d!} \Bigg( \sum_{q\in\PP_k\setminus \PP_j} \sum_{a=1}^\infty \frac{1}{q^{as}} \Bigg)^d
\Bigg] \sum_{m\in T_j} \frac{1}{m^s}.
\]
Using \eqref{sumq} multiple times, we see that the first double sum on the right side is
at most
\[
  \sum_{q\in \PP_k\setminus \PP_j} \frac{1}{q^s(q^s-1)} \le \sum_{q\in \PP_k\setminus \PP_j} \frac{\lam(s)}{(q-1)^{2s}} \le \lam(s) \( V_k(s) - V_j(s) \)^2,
\]
the second double sum over $q$ and $a$ is at most
\[
 \sum_{q\in \PP_k \setminus \PP_j} \frac{\lam(s)}{(q-1)^s} = \lam(s) \( V_k(s) - V_j(s) \),
\]
and the sum on $m$ is bounded above by
\[
 \prod_{p\in \PP_j} (1-1/p^s)^{-1} \le \er^{\lam(s) V_j(s)}.
\]
Thus,
\[
 \sum_{n\in \Tt_{j,k}} \frac{1}{n^s}  \le \er^{\lam(s) V_j(s)} \( V_k(s)-V_j(s) \)^2 
\Bigg[ \lam(s) + \lam(s)^2\sum_{d=2}^\infty \frac{(V_k(s)-V_j(s))^{d-2} \lam(s)^{d-2}}{d!} \Bigg].
\]
Finally, $d! > (d-2)!$ and so the sum on $d$ is less than $\er^{\lam(s) (V_k(s)-V_j(s))}$.
\end{proof}

We now come to the main iteration inequality.  Instead of descending just one level as in
the proof of Lemma \ref{V0} (that is, examining the prime factors of $p-1$),  we descend a finite 
(and bounded) number of levels, 
examining the prime factors $q_1$ of $p-1$, the prime factors $q_2$ of each $q_1-1$, etc.
To state our result, we introduce a family of matrices $M_{s,j,Q}$.  Let 
\be\label{UQ}
U_Q = \{ 1\le n\le Q : (n,Q)=1 \text{ and } \forall p|Q \text{ such that } p\not\in\PP, n\not\equiv1\pmod{p} \}.
\ee
By \eqref{divisorclosed}, for any $Q$ and $p\in \PP$ with $p\nmid Q$, we have $p\mod Q \in U_Q$.
For $j\ge 1$, $s>0$ and $Q\in\NN$, let $M_{s,j,Q}$ be the $Q \times Q$ matrix
whose entries are given by
\be\label{MsjQ}
 M_{s,j,Q} (a,b) = \ssum{m\in T_j \\ am\equiv b\pmod{Q}} m^{-s}
\ee
if $a\in U_Q$ and $b\in U_Q$, and $M_{s,j,Q}(a,b)=0$ otherwise.  For a generic square matrix $M$ with non-negative
entries, we introduce notation for row sums and column sums:
\[
 R_a(M) = \sum_b M(a,b), \quad R(M) = \max_a R_a(M), \qquad C_b(M)=\sum_a M(a,b), \quad C(M)=\max_b C_b(M).
\]

\begin{lem}\label{V2}
 Suppose that $n\in \NN$, $h>j\ge n$, $Q\in\NN$ and $\PP_{j-n}$ contains every prime in $\PP$ which 
divides $Q$.  Then, for $M=M_{s,j,Q}$, 
\[
 V_h(s) \le V_j(s) + \sum_{q\in \PP_{h-n}\setminus \PP_{j-n}} \frac{R_{q\!\!\!\!\mod Q}(M^n)}{q^s} + 2n\lam(s)^2
\er^{n\lam(s) V_{h-1}(s)} \( V_{h-1}(s)-V_{j-n}(s)\)^2.
\]
\end{lem}

\begin{proof}
 We'll first show, by induction on $n$, that for any integers $h$ and $j$ satifying $h>j\ge n$,
\begin{multline}\label{vhvj1}
V_h(s) \le V_j(s)+2n\lam(s)^2 \er^{n\lam(s) V_{h-1}(s)} \(V_{h-1}(s)-V_{j-n}(s)\)^2 \\
+\sum_{q_n\in\PP_{h-n}\setminus \PP_{j-n}} \frac{1}{q_n^s} \!
\ssum{m_n\in T_{j-n}\\m_nq_n+1=q_{n-1}\\ q_{n-1}\text{ prime}} \! \frac{1}{m_n^s} \!
\ssum{m_{n-1}\in T_{j-n+1}\\m_{n-1}q_{n-1}+1=q_{n-2}\\q_{n-2}\text{ prime}} \!\! \frac{1}{m_{n-1}^s} \cdots \!\!
 \ssum{\mone \in T_{j-1}\\ \mone\qone+1=\qz\\ \qz\text{ prime}} \frac{1}{m_1^s}.
\end{multline}
To begin the induction, we use \eqref{phth-1} and Lemma \ref{V1} to obtain
\bal
V_h(s) &= V_j(s) + \sum_{\qz-1\in T_{h-1}\setminus T_{j-1}} \frac{1}{(\qz-1)^s} \\
&\le V_j(s) + 2\lam(s)^2 \er^{\lam(s)V_{h-1}(s)}\(V_{h-1}(s)-V_{j-1}(s)\)^2 + \sum_{\qz-1\in\Tb_{j-1,h-1}} \frac{1}{(\qz-1)^s}.
\eal
In the final sum, we may write $\qz=1+\mone \qone$, where $\mone\in T_{j-1}$ and 
$\qone\in \PP_{h-1}\setminus \PP_{j-1}$.  This proves \eqref{vhvj1} when $n=1$.

Now suppose that \eqref{vhvj1} holds for some $n$, and assume that $h>j\ge n+1$. 
In the multiple sum in \eqref{vhvj1}, replace $q_n^{-s}$ with $(q_n-1)^{-s}$ and observe that
 $q_n-1 \in T_{h-n-1} \setminus T_{j-n-1}$.  The contribution to the multiple sum from those
summands with $q_n-1 \in \Tt_{j-n-1,h-n-1}$ is, by Lemma \ref{V1} and \eqref{VhVh-1}, at most
\bal
2\lam(s)^2& \er^{\lam(s) V_{h-n-1}(s)} \( V_{h-n-1}(s)-V_{j-n-1}(s)\)^2 \sum_{m_n\in T_{j-n}} m_n^{-s} \cdots
\sum_{\mone \in T_{j-1}} m_1^{-s} \\
&\le 2\lam(s)^2 \er^{\lam(s) V_{h-1}(s)} \( V_{h-1}(s)-V_{j-n-1}(s)\)^2 \er^{\lam(s) (V_{j-n}(s)+\cdots+V_{j-1}(s))} \\
&\le 2\lam(s)^2 \er^{\lam(s) (n+1) V_{h-1}(s)} \( V_{h-1}(s)-V_{j-n-1}(s)\)^2.
\eal
If $q_n-1\in \Tb_{j-n-1,h-n-1}$, then $q_n=1+q_{n+1}m_{n+1}$, where
$q_{n+1}\in\PP_{h-n-1}\setminus\PP_{j-n-1}$ and $m_{n+1}\in T_{j-n-1}$.  This proves \eqref{vhvj1} with
$n$ replaced by $n+1$.  By induction, \eqref{vhvj1} follows for all $n$.

In \eqref{vhvj1}, we enlarge the range of all sums on $m_i$ to $m_i\in T_{j-1}$.  Also,
for $0\le i\le n-1$,
we relax the condition that $q_i$ is prime to $q_i\equiv a_i\pmod{Q}$, where $a_i\in U_Q$.
Recalling \eqref{MsjQ}, we find that the multiple sum in \eqref{vhvj1} is at most
\bal
&\sum_{q_n\in\PP_{h-n}\setminus\PP_{j-n}} q_n^{-s} \sum_{a_{n-1}\in U_Q} \ssum{m_n\in T_{j-1} \\ m_nq_n+1\equiv a_{n-1}\pmod{Q}} m_n^{-s}
\cdots  \sum_{\az\in U_Q}  \ssum{\mone\in T_{j-1} \\ \mone \aone+1\equiv \az\pmod{Q}} m_1^{-s} \\
&= \sum_{q_n\in\PP_{h-n}\setminus\PP_{j-n}} q_n^{-s} \sum_{a_{n-1},\ldots,\az\in U_Q} M(q_n\!\!\!\!\mod Q,a_{n-1}) M(a_{n-1},a_{n-2}) \cdots
M(\aone,\az) \\ &= \sum_{q_n\in\PP_{h-n}\setminus\PP_{j-n}} q_n^{-s} R_{q_n^{\phantom{a}}\!\!\!\!\!\mod{Q}}(M^n).
\eal
This completes the proof of the lemma.
\end{proof}

Assuming $V_\infty(s)$ exists and $h$ and $j$ are large, $V_{h-1}(s)-V_{j-n}(s)$
will be very small if $j$ is large.  Consequently, of the three terms
on the right side of the inequality in Lemma \ref{V2},
the third may be regarded as ``small'', since it is quadratic in $V_{h-1}(s)-V_{j-n}(s)$.
The second term is at most $(V_{h-1}(s)-V_{j-n}(s)) R(M^n)$, and
can be regarded as larger than the third term.  It can also be made very small, provided that
$M$ is a contracting matrix (all eigenvalues lie inside the unit circle) and $n$ is large enough.
Under this assumption on $M$, it follows that
\[
 V_h(s) \le V_{j}(s) + (V_{h-1}(s)-V_{j-n}(s)) \eps,
\]
where $\eps$ is small.  Iteration of this inequality, with $j$ and $n$ fixed,
then shows that the sequence $V_0(s), V_1(s), \ldots$
is bounded.  The next lemma makes this heuristic precise.

\begin{lem}\label{iterlem}
 Suppose that for some $y$ and for $Q=\prod_{p\le y} p$, $M_{1,\infty,Q}$ is a contracting matrix.
Then for some $s<1$, $V_{\infty}(s)$ is finite.
\end{lem}

\begin{proof}
 By assumption, $R(M_{1,\infty,Q}^n) \le \frac14$ for some $n$.
Let $D=V_\infty(1)$ ($D$ exists by Proposition \ref{simplesieve}) and let
\[
 \eps = \frac{1}{100 n \er^{2n(D+2)}}.
\]
Fix $j$ large enough so that $j>n$, $\PP_{j-n}$ contains all primes in $\PP$ which are $\le y$,
and  $V_j(1) - V_{j-n}(1) \le \eps/2$.
By Lemma \ref{V0}, $V_j(s)$ and $V_{j-n}(s)$ are continuous for $0<s\le 1$, as are 
all entries of $M_{s,j,Q}$.  Note that $R(M_{1,j,Q}^n)\le R(M_{1,\infty,Q}^n) \le \frac14$.
 Therefore, there is an $s\in [0.9,1)$ such that
\begin{itemize}
 \item[(a)] $V_j(s) \le D+1$,
 \item[(b)] $V_{j}(s) - V_{j-n}(s) \le \eps$,
 \item[(c)] $R(M_{s,j,Q}^n) \le \frac13$.
\end{itemize}
Since $s\ge 0.9$, we have $\lam(s)\le 2$.  By Lemma \ref{V2} and (c), for any $h>j$, it follows that
\[
V_h(s) \le V_{j}(s)+\frac13 \( V_{h-1}(s)-V_{j-n}(s)\)+8n\er^{2n V_{h-1}(s)} \( V_{h-1}(s)-V_{j-n}(s)\)^2.
\]
For $k\ge 0$, let $x_k=V_{j+k}(s)-V_j(s)$.  Then $x_0=0$ and, by (a) and (b), for $k\ge 1$ we have
\bal
 x_{k} &\le \frac13 (x_{k-1}+V_{j}(s)-V_{j-n}(s)) + 8n\er^{2n(x_{k-1}+V_j(s))} \( x_{k-1} +V_{j}(s)-V_{j-n}(s)\)^2 \\ 
&\le \frac13 (x_{k-1}+\eps)+8n \er^{2n(D+1+x_{k-1})}(x_{k-1}+\eps)^2 =: f(x_{k-1}).
\eal
We have $f(0)>0$, $f''(x)>0$ for $x>0$ and
\[
 f(\eps) = \frac23 \eps + 8n\er^{2n(D+1+\eps)}(2\eps)^2 < \frac23 \eps + \frac{32}{100}\eps < \eps.
\]
Therefore, $f(x)=x$ has a unique root $\tilde{x} \in (0,\eps)$ and it follows that 
$\lim_{k\to\infty} x_k \le \tilde{x}$.  Consequently, 
\[
V_\infty(s)\le V_{j}(s)+\tilde{x} \le D+1+\eps. \qedhere
\]
\end{proof}

%%%%%%%%%%%%%%%%%%%%%%%%%%%%%%%%%%%%%%%%%%%%%%%%%%%%%%
%
%
\section{Matrix eigenvalues and the proof of Theorem \ref{mainthm}}
%
%
%%%%%%%%%%%%%%%%%%%%%%%%%%%%%%%%%%%%%%%%%%%%%%%%%%%%%%

Throughout this section, we assume that $Q=\prod_{p\le y} p$ and $M=M_{1,\infty,Q}$.  Observe that by Proposition 
\ref{simplesieve},
\be\label{K}
K = \prod_{p\in\PP} \(1-\frac{1}{p} \) \gg 1.
\ee
Because all entries of $M$ are nonnegative, the Perron-Frobenius theorem implies that
there is an eigenvalue of largest modulus which is real and positive.
The matrices $M$ are similar to the matrices studied
in \cite[\S 2]{FKL}, and we will likewise focus on bounding column sums of $M$.
  However, the estimation problem is much more
complicated than the analogous problem in \cite{FKL}.  

\begin{lem}\label{Cb}
 For any $b\in U_Q$, let $d=(b-1,Q)$ and $b'=\frac{b-1}{d}$.  Then
\be\label{CbM}
 C_b(M) = \frac{\phi(d)}{d} \ssum{k\in T_\infty \\ (k,Q/d)=1 \\ \eqref{kb'}} \frac{1}{k},
\ee
where
\be\label{kb'}
\forall p\le y \text{ with } p\not\in\PP, k\not\equiv b'\pmod{p}.
\ee
\end{lem}

\begin{proof}
 By the definition of $U_Q$ in \eqref{UQ}, $2|d$ and for all $p|d$, $p\in \PP$.  In \eqref{MsjQ}, therefore,
 $am+1\equiv b\pmod{Q}$
implies that $d|m$.  Writing $m=dk$, we have $(k,Q/d)=1$ and $ak\equiv b'\pmod{Q/d}$.
Since $a\in U_Q$, $a\not\equiv 1\pmod{p}$ for any $p\le y$ with $p\not\in \PP$.  Hence, \eqref{kb'} holds.
Therefore, by \eqref{MsjQ},
\[
 C_b(M) = \sum_{a\in U_Q} \ssum{k\in T_\infty\\ak\equiv b'\pmod{Q/d}\\ \eqref{kb'}} \frac{1}{dk} = 
\frac{1}{d} \ssum{k\in T_\infty\\ (k,Q/d)=1 \\ \eqref{kb'}} \frac{1}{k} \# \{a\in U_Q : ak\equiv
b' \pmod{Q/d} \}.
\]
For every $k\in T_\infty$ satisfying $(k,Q/d)=1$ and \eqref{kb'}, there is a unique 
solution $a\mod Q/d$ of the congruence $ak\equiv b'\pmod{Q/d}$ and moreover this solutions satisfies
$a\in U_Q$.  Thus, there are $\phi(d)$ solutions 
$a \in U_Q$, and this completes the proof.
\end{proof}

Notice that if we ignore condition \eqref{kb'}, then we obtain from \eqref{CbM} the upper bound
\be\label{crude}
C_b(M) \le \frac{\phi(d)}{d} \ssum{k\in T_\infty \\ (k,Q/d)=1} \frac{1}{k} =
\frac{\phi(d)}{d} \sprod{p\in \PP \\ p|d\text{ or }p>y} \(1-\frac{1}{p}\)^{-1} = \prod_{p\in\PP,p>y}
 \(1-\frac{1}{p}\)^{-1}.
\ee
The product on the far right side of \eqref{crude} is always greater than 1, however it tends to
1 as $y\to\infty$ by Proposition \ref{simplesieve}.
In order to obtain a bound $C(M)<1$, it is necessary to use \eqref{kb'} to eliminate
some numbers $k$ from the sum in \eqref{crude}.  However, if
$k\in T_\infty$ with $(k,Q/d)=1$ and $k<y$, then only primes dividing $d$ may divide $k$.  
In the worst case $d=2$, the only numbers $k<y$ that are available to eliminate are powers of 2.
If there is a prime
$p\not\in\PP$ for which 2 is a primitive root (generator of $(\ZZ/p\ZZ)^*$), then we will succeed.

\begin{lem}\label{2primroot}
 Suppose $p\not\in \PP$ and 2 is a primitive root of $p$.  Then, for large enough $y$ depending on $p$,
\[
 C(M) \le 1 - 2^{1-p} K,
\]
where $K$ is defined in \eqref{K}.
\end{lem}

\begin{proof}
 For any $b\in U_Q$, let $d=(b-1,Q)$ and $b'=\frac{b-1}{d}$ as before.  Since $b\not\equiv 1\pmod{p}$,
 we have $(b',p)=1$.  Hence, there is an exponent $\theta\in\{0,1,\ldots,p-2\}$ such that $2^\theta\equiv b'\pmod{p}$.
By Lemma \ref{Cb} and \eqref{crude},
\[
 C_b(M)\le \frac{\phi(d)}{d} \Bigg[ \ssum{k\in T_\infty \\ (k,Q/d)=1} \frac{1}{k}
 - \frac{1}{2^\theta} \Bigg] \le \sprod{p\in\PP \\ p>y}  \(1-\frac{1}{p}\)^{-1} - 2^{2-p} \frac{\phi(d)}{d}.
\]
The lemma follows upon observing that
\[
 \inf_{d\in T_\infty} \frac{\phi(d)}{d} = K
\]
and that if $y$ is large then
\[
 \sprod{p\in\PP \\ p>y}  \(1-\frac{1}{p}\)^{-1} \le 1 + 2^{1-p} K. \qedhere
\]
\end{proof}

\textbf{Remarks.}  It is conjectured that there are infinitely many primes that have 2 as a 
primitive root, but this is an open problem.   Hooley \cite{Ho}
showed that the Riemann Hypothesis for the Dedekind zeta functions $\zeta_{K_r}(s)$ for the number fields 
$K_r=\QQ(2^{1/r},\er^{2\pi i/r})$, where $r$ runs over the primes, implies that the number of
primes $p\le x$ which have 2 as a primitive root is $\sim cx/\log x$, where 
$c=\prod_r ( 1 - \frac{1}{r(r-1)} ) = 0.3739\ldots$.  This asymptotic formula is known
as Artin's primitive root conjecture for the base 2.  If true, then by Proposition \ref{simplesieve},
most of these primes are not in $\PP$, and we obtain Theorem \ref{mainthm} upon invoking Lemma \ref{2primroot}.
For more about Artin's conjecture, the reader may consult the comprehensive survey article
\cite{M}.

Unconditionally, Lemmas \ref{iterlem} and \ref{2primroot} imply Theorem 1 
in the case that 2 is a primitive root of $\pz$ ($\pz\in\{3,5,11,13,19,\ldots\}$), or if there is a prime
$q\equiv 1\pmod{\pz}$ with 2 as a primitive root; for example if $\pz=7$ then we may take $q=29$.

\bigskip

There is a way around invoking Artin's conjecture: by examining column sums of small powers of $M$,
we succeed if there is a prime $p\not\in\PP$ with $(\ZZ/p\ZZ)^*$ generated by a bounded set of small primes.
The following result of  Gupta and Murty \cite{GM} supplies us with the necessary prime $p$.

\begin{lem}\label{GuptaMurty}
 For $\gg x/\log^2 x$ primes $p\le x$, $(\ZZ/p\ZZ)^*$ is generated by 2, 3 and 5.
\end{lem}

\textbf{Remarks.}  Heath-Brown \cite{HB} proved the stronger statement that for  $\gg x/\log^2 x$ primes $p\le x$,
either 2, 3 or 5 is a primitive root of $p$.  
Our argument below, in fact, requires only the weaker statement that for some $k$ and primes
$\pone,\ldots,p_k$, each with 2 as a primitive root, there are $\gg x/\log^2 x$ primes $p\le x$
for which $(\ZZ/p\ZZ)^*$ is generated by $2,\pone,\ldots,p_k$.
We would then iterate Lemma \ref{CbMk} below $k$ times instead of twice.

\bigskip

Utilizing Lemma \ref{GuptaMurty}, we will show that $C(M^3)<1$ for large $y$.
Our main tool is the following, which roughly says that if $C_b(M^k)<1$ for every $b$
lying in some  arithmetic progression, then $C_b(M^{k+1})<1$ for all $b$ lying in a larger arithmetic progression.

\begin{lem}\label{CbMk}
 Let $p$ be a prime in $\PP$ with 2 as a primitive root, and let $n\in T_\infty$ satisfy $n|Q$ and $p\nmid n$.
Let $u\in \NN$.
Suppose that for large $y$ and for all $b\equiv 1\pmod{pn}$, $C_b(M^u)\le 1-\delta$ where $\del>0$.
Then, for large enough $y$ (depending on $p, n, \del, \pz, u$) and all $b\equiv 1\pmod{n}$,
$C_b(M^{u+1}) \le 1-\del'$, where
\[
 \del' = \frac{\del K}{2^p_{\phantom{a}} n}.
\]
\end{lem}

\begin{proof}
 Suppose that $b\in U_Q$ with $b\equiv 1\pmod{n}$.  If $p|(b-1)$, we apply \eqref{crude} 
and the general inequality $C_b(AB)\le C(A) C_b(B)$ to obtain
\[
 C_b(M^{u+1}) \le C_b(M^u) C(M) \le (1-\del) \prod_{p\in\PP,p>y}
 \(1-\frac{1}{p}\)^{-1} \le 1 - \frac{\del}{2} \le 1 - \del'
\]
if $y$ is large enough.  Now assume $p\nmid (b-1)$.  As in Lemma \ref{Cb}, put 
$ d = (b-1,Q)$ and $b' = \frac{b-1}{d}.$  We have
\be\label{Cak}
\begin{split}
 C_b(M^{u+1}) &= \sum_{a\in U_Q} C_a(M^u) M(a,b) \\
&= \sum_{a\in U_Q} C_a(M^u) \ssum{m\in T_\infty \\ am+1\equiv b\pmod{Q}} \frac{1}{m} \\
&= \frac1d \ssum{k\in T_\infty \\ (k,Q/d)=1 \\ \eqref{kb'}} \frac{1}{k} \ssum{a\in U_Q \\ ak\equiv b'\pmod{Q/d} \\
} C_a(M^u).
\end{split}\ee
For each $k$, the congruence $ak\equiv b'\pmod{Q/d}$ has a unique solution $a\mod Q/d$, hence there are
$\phi(d)$ solutions $a\in U_Q$.
By assumption, there is a $\theta\in\{0,1,\ldots,p-2\}$ with $2^\theta\equiv b'\pmod{p}$.  
In \eqref{Cak}, we use the crude bound $C_a(M^u)\le C(M^u)\le C(M)^u$ for all pairs $a,k$ except
when both $k=2^\theta$ and $a\equiv 1\pmod{n}$. In the latter case, $a\equiv 1\pmod{p}$ as well,
hence $a\equiv 1\pmod{pn}$ and $C_a(M^u) \le 1-\del$.   Also, since $n|Q$ and $b\equiv 1\pmod{n}$,
we have $n|d$.  By \eqref{crude} and \eqref{Cak}, 
\bal
C_b(M^{u+1})&\le C(M)^u \Bigg[ \frac{\phi(d)}{d} \ssum{k\in T_\infty \\ (k,Q/d)=1} \frac{1}{k} -
\frac{1}{2^\theta d} \ssum{a\in U_Q \\ a\equiv 1\pmod{nQ/d}} 1 \Bigg] + 
\frac{1}{2^\theta d} \ssum{a\in U_Q \\ a\equiv 1\pmod{nQ/d}} (1-\del) \\
&\le \max\(1, C(M)^u\) \frac{\phi(d)}{d} \ssum{k\in T_\infty \\ (k,Q/d)=1} \frac{1}{k}
-\frac{\del}{2^\theta d} \phi\pfrac{d}{n}\\
&\le \prod_{p\in\PP,p>y} \(1-\frac{1}{p}\)^{-(u+1)} - \frac{\del}{2^{p-2}_{\phantom{a}} d} \phi\pfrac{d}{n}.
\eal
Since $\phi(d/n) \ge \phi(d)/n$ and $\phi(d)/d \ge K$, upon recalling the definition of $\del'$
we conclude that
\[
C_b(M^{u+1}) \le \prod_{p\in\PP,p>y} \(1-\frac{1}{p}\)^{-(u+1)} - \frac{\del K}{2^{p-2}_{\phantom{a}}n} \le 1-\del'
\]
if $y$ is large enough.
\end{proof}

\begin{proof}[Proof of Theorem \ref{mainthm}]
If $p_0\in\{3,5\}$, Lemma \ref{2primroot} (with $p=p_0$) implies that $C(M)<1$ for large enough $y$.
Hence, by Lemma \ref{iterlem}, $V_\infty(s)$ is finite for some $s<1$.

\newcommand{\numexp}[1]{\ensuremath{#1_{\phantom{a}}^{a_#1^{\phantom{a}}}}}
\newcommand{\thirty}{\ensuremath{\numexp{2}\numexp{3}\numexp{5}}}

Now assume that $3\in\PP$ and $5\in\PP$.
Combining Lemmas \ref{sieve3} and \ref{GuptaMurty}, we find that there
is a prime $\pone\not\in \PP$ for which 2,3 and 5 generate $(\ZZ/\pone\ZZ)^*$.
Following the proof of Lemma \ref{2primroot}, for any $b\in U_Q$ with $b\equiv 1\pmod{30}$, there are
exponents $a_2^{\phantom{a}},a_3^{\phantom{a}},a_5^{\phantom{a}}\in\{0,1,\ldots,\pone-2\}$ so that 
$\thirty\equiv b'\pmod{\pone}$.  As before, $b'=\frac{b-1}{(b-1,Q)}$.
By \eqref{kb'}, $k=\thirty$ is excluded from the sum in \eqref{CbM}.  By 
\eqref{crude}, if $y$ is large enough then
\bal
C_b(M) \le \prod_{p\in\PP,p>y} \(1-\frac{1}{p}\)^{-1} - \frac{\phi(d)/d}{\thirty}
\le \prod_{p\in\PP,p>y} \(1-\frac{1}{p}\)^{-1} - \frac{K}{30^{\pone-2}_{\phantom{1}}} < 1 - \del,
\eal
where $\del = K/30^{p_1^{\phantom{a}}-1}_{\phantom{1}}$.  By Lemma \ref{CbMk} with $n=6$, $p=5$
and $u=1$, we find that for large enough $y$, $C_b(M^2)\le 1-\del'$ for every $b\equiv 1\pmod{6}$, where
\[
 \del' = \frac{K \del}{2^5\cdot 6}.
\]
A second application of Lemma \ref{CbMk}, with $n=2$, $p=3$ and $u=2$ implies that for \emph{every} 
$b\in U_Q$, $C_b(M^3)\le 1-\del''$ if $y$ is large enough, where $\del''=K \del'/16$.  Thus, the 
dominant eigenvalue of $M^3$ is
at most $1-\del''$, hence the dominant eigenvalue of $M$ is $\le (1-\del'')^{1/3} < 1$. 
 Finally, applying Lemma \ref{iterlem}, we find that $V_\infty(s)$ is finite for some $s<1$.
It follows immediately that $P(x)=O(X^s)$.
\end{proof}

\bigskip

\noindent
{\bf Acknowledgements.}  The author thanks Paul Pollack for helpful conversations concerning
Lemma \ref{sieve3}, and is thankful to Paul Bateman for introducing to him Carmichael's conjecture and
related problems.

The author's research was supported by National Science Foundation Grants DMS-0901339
and DMS-1201442.  Much of the work was accomplished while the author attended the N.S.F. supported 
Workshop in Linear Analysis and Probability at Texas A\&M University, July-August, 2012.

%%%%%%%%%%%%%%%%%%%%%%%%%


\begin{thebibliography}{99}

\bibitem{Bay}  J. Bayless, {\it The Lucas-Pratt primality tree},
Math. Comp.  {\bf 77}  (2008), 495-502.

\bibitem{C1} R. D. Carmichael , {\it  On Euler's $\phi$-function
},  Bull. Amer. Math. Soc. {\bf  13 } (1907) , 241--243. 

\bibitem{C2} \bysame , {\it  Note on Euler's $\phi$-function
},  Bull. Amer. Math. Soc. {\bf  28 } (1922) , 109--110.

\bibitem{EGPS} P. Erd\H os, A. Granville, C. Pomerance, and C. Spiro,
{\it On the normal behavior of the iterates of some arithmetic functions},
in Analytic Number Theory, Proceedings of a conference in honor of Paul T.
Bateman, Birkh\" auser, Boston, 1990, 165--204.

\bibitem{FJ}  C. Finch and L. Jones, {\it A Curious Connection Between
Fermat Numbers and Finite Groups}, Amer. Math. Monthly {\bf 109} (2002), 517--524.

\bibitem{Ford98} K. Ford, {\it The distribution of totients},
Ramanujan J. (Paul Erd\H os memorial issue) {\bf 2} (1998), 67--151.

\bibitem{FKL} K. Ford, S. Konyagin and F. Luca, {\it Prime chains and Pratt trees},
Geom. Funct. Anal. {\bf 20} (2010), 1231--1258.

\bibitem{FKL2} K. Ford, S. Konyagin and F. Luca, {\it On groups with perfect
order subsets}, Moscow J. Comb. Number Theory {\bf 2}, no. 4 (2012), to appear.

\bibitem{GM}  R. Gupta and M. Ram Murty, {\it A remark on Artin's conjecture},
Inv. Math. {\bf 78} (1984), 127--130.

\bibitem{HR} H. Halberstam and H.-E. Richert, {\it Sieve methods\/},
Academic Press, London, 1974.

\bibitem{HB} D. R. Heath-Brown, {\it Artin's conjecture for primitive roots}, Quart.
J. Math. Oxford (2) {\bf 37} (1986), 27--38.

\bibitem{Ho} C. Hooley, {\it On Artin's conjecture},  J. Reine
  Angew. Math.  {\bf 225}  (1967), 209--220.

\bibitem{LP} F. Luca and C. Pomerance, {\it Irreducible radical
  extensions and Euler-function chains}, in ``Combinatorial number
  theory'', 351--361, de Gruyter, Berlin (2007).

\bibitem{M} P. Moree, {\it Artin's primitive root conjecture - a survey},
INTEGERS (The electronic journal of combinatorial number theory) {\bf 12A} (2012),
paper A13.  See also \texttt{arXiv.math/0412262v2} (2012).

\bibitem{Pom74} C. Pomerance, {\it On Carmichael's conjecture}, Proc. Amer. Math. Soc.
{\bf 43} (1974), 297--298.

\bibitem{Ten} G. Tenenbaum, {\it Introduction \`a la th\'eorie analytique et probabiliste des nombres,}
%R deuxi\`eme \'edition, Cours Sp\'ecialis\'e, no. 1, Soci\'et\'e Math\'ematique de France (1995)^M
troisi\`eme \'Edition, coll. \'Echelles, Belin, 2008, 592 pp.
English translation of the second edition: {\it Introduction to analytic and probabilistic number theory\/}, 
Cambridge Univ. Press, 1995.


\end{thebibliography}
\end{document}